\documentclass[14pt]{amsart}

\usepackage{amsmath,arydshln,multirow}

\usepackage{arydshln}
\usepackage{cases}
\usepackage{amsmath}
\usepackage{amsfonts}
\usepackage{bm}
\usepackage{arydshln}
\usepackage{amsfonts,amsmath,amssymb,amscd,bbm,amsthm,mathrsfs,dsfont}
\usepackage{mathrsfs}
\usepackage{pb-diagram}
\usepackage{amssymb}
\usepackage{xypic}
\usepackage{mathtools}
\newtheorem{Theorem}{Theorem}[section]
\newtheorem{Lemma}[Theorem]{Lemma}

\newtheorem{Definition}[Theorem]{Definition}
\newtheorem{Corollary}[Theorem]{Corollary}
\newtheorem{Proposition}[Theorem]{Proposition}

\newtheorem{Remark}[Theorem]{Remark}
\newtheorem{Conjecture}[Theorem]{Conjecture}

 \normalbaselines
\title{Unistructurality of cluster algebras}

\author{Peigen Cao $\;\;\;\;\;\;$ Fang Li $\;\;\;\;\;\;$}
\address{Peigen Cao
\newline Department
of Mathematics, Zhejiang University (Yuquan Campus), Hangzhou, Zhejiang
310027,  P.R.China}
\email{peigencao@126.com}

\address{Fang Li
\newline Department
of Mathematics, Zhejiang University (Yuquan Campus), Hangzhou, Zhejiang
310027, P.R.China}
\email{fangli@zju.edu.cn}

\begin{document}
%\begin{CJK*}{GBK}{song}
\renewcommand{\thefootnote}{\alph{footnote}}

\setcounter{footnote}{-1} \footnote{\emph{ Mathematics Subject
Classification(2010)}: 13F60, 05E40}
\renewcommand{\thefootnote}{\alph{footnote}}
\setcounter{footnote}{-1} \footnote{ \emph{Keywords}: Cluster algebras,
Unistructurality}

\begin{abstract}
We prove that any skew-symmetrizable cluster algebra is unistructural, which is a conjecture by  Assem,  Schiffler, and  Shramchenko. As a corollary, we obtain that a cluster automorphism of a cluster algebra $\mathcal A(\mathcal S)$ is just an automorphism of the ambient field $\mathcal F$ which restricts to a permutation of cluster variables of $\mathcal A(S)$.
\end{abstract}

\maketitle
\bigskip

\section{introduction}
A cluster algebra $\mathcal A(\mathcal S)$ is a subalgebra of an ambient field $\mathcal F$ generated by certain combinatorially defined generators (i.e.,  {\em cluster variables}), which are grouped into overlapping {\em clusters}. Roughly speaking, a cluster algebra is a commutative algebra with an extra combinatorial structure.
Assem,  Schiffler, and  Shramchenko conjectured that

\begin{Conjecture}\label{conjecture}
(\cite[Conjecture 1.2]{ASV})  The set of cluster variables uniquely determines  the
cluster algebra structure, i.e., any cluster algebra is unistructural (see Definition \ref{maindef} for details).
\end{Conjecture}
This is a very interesting conjecture in the following sense. We know that a cluster algebra $\mathcal A(\mathcal S)$ is a  commutative algebra with an extra combinatorial structure. As an algebra, $\mathcal A(\mathcal S)$ is generated by the set of cluster variables. So the algebraic structure of $\mathcal A(\mathcal S)$ is uniquely determined by the set of cluster variables. The above conjecture predicts that the combinatorial structure of $\mathcal A(\mathcal S)$ can be also uniquely determined by the set of cluster variables.

 Conjecture \ref{conjecture} has been affirmed for cluster algebras of Dynkin type or rank $2$  in \cite{ASV}, for cluster algebras of type $\tilde A$  in \cite{B}.   Recently, Bazier-Matte and Plamondon have affirmed Conjecture \ref{conjecture} for cluster algebras from surfaces without punctures in \cite{BP}.  Note that all cluster algebras considered above are {\em with trivial coefficients}.

In this paper,  we will affirm  Conjecture \ref{conjecture} for  any skew-symmetrizable cluster algebras {\em with general coefficients}. The following is our main result:

\begin{Theorem}\label{mainthm}
Any skew-symmetrizable cluster algebra is unistructural.
\end{Theorem}

This paper is organized as follows. In Section 2, some basic definitions, notations and known results are introduced. In Section 3, the proof of Theorem \ref{mainthm} is given.

\section{Preliminaries}

\subsection{Cluster algebras}
Recall that $(\mathbb P, \oplus, \cdot)$ is a {\bf semifield} if $(\mathbb P,  \cdot)$ is an abelian multiplicative group endowed with a binary operation of auxiliary addition $\oplus$ which is commutative, associative, and distributive with respect to the multiplication $\cdot$ in $\mathbb P$.

Let $Trop(y_1,\cdots,y_m)$ be a free abelian group generated by $\{y_1,\cdots,y_m\}$. We define the addition $\oplus$ in $Trop(y_1,\cdots,y_m)$ by $\prod\limits_{i}y_i^{a_i}\oplus\prod\limits_{i}y_i^{b_i}=\prod\limits_{i}y_i^{min(a_i,b_i)}$, then $(Trop(y_1,\cdots,y_m), \oplus)$ is a semifield, which is called a {\bf tropical semifield}.

The multiplicative group of any semifield $\mathbb P$ is torsion-free \cite{FZ}, hence its group ring $\mathbb Z\mathbb P$ is a domain.
We take an ambient field $\mathcal F$  to be the field of rational functions in $n$ independent variables with coefficients in $\mathbb Z\mathbb P$.

An integer matrix $B_{n\times n}=(b_{ij})$ is  called  {\bf skew-symmetrizable} if there is a positive integer diagonal matrix $S$ such that $SB$ is skew-symmetric, where $S$ is said to be a {\bf skew-symmetrizer} of $B$.

\begin{Definition}
A {\bf  seed} in $\mathcal F$ is a triplet $({\bf x},{\bf y},B)$ such that

(i)  ${\bf x}=\{x_1,\dots, x_n\}$ is a free generating set of $\mathcal F$ over $\mathbb{ZP}$. We call  ${\bf x}$   the {\bf cluster} and $x_1,\dots,x_n$ the {\bf cluster variables}  of  $({\bf x},{\bf y},B)$.

(ii) ${\bf y}=\{y_1,\cdots,y_n\}$ is a subset of  $\mathbb P$. We call $y_1,\cdots,y_n$ the {\bf coefficients} of  $({\bf x},{\bf y},B)$.

(iii) $B=(b_{ij})$ is a skew-symmetrizable matrix, called an {\bf exchange matrix}.
\end{Definition}

Let $({\bf x},{\bf y}, B)$ be a  seed in $\mathcal F$, one can associate {\em binomials} $F_1,\cdots,F_n$ defined by
\begin{eqnarray}
F_k=\frac{y_k}{1\oplus y_k}\prod\limits_{b_{ik}>0}x_i^{b_{ik}}+\frac{y_k}{1\oplus y_k}\prod\limits_{b_{ik}<0}x_i^{-b_{ik}}.\nonumber
\end{eqnarray}
$F_1,\cdots,F_n$ are called the {\bf exchange binomials} of $({\bf x},{\bf y}, B)$.

\begin{Definition}\label{defmutation}
Let $({\bf x},{\bf y},B)$ be a  seed in $\mathcal F$, and $F_1,\cdots, F_n$ be the exchange polynomials of  $({\bf x},{\bf y},B)$. Define the {\bf mutation}  of  $({\bf x},{\bf y},B)$ in the direction $k\in\{1,\cdots,n\}$ as a new triple $\mu_k({\bf x},{\bf y},B)=( {\bf x}^{\prime},  {\bf y}^{\prime},  B^{\prime})$ in $\mathcal F$
given by
\begin{eqnarray}
b_{ij}^{\prime}&=&\begin{cases}-b_{ij}~,& i=k\text{ or } j=k;\\ b_{ij}+sgn(b_{ik})max(b_{ik}b_{kj},0)~,&otherwise.\end{cases}\nonumber\\
x_i^{\prime}&=&\begin{cases}x_i~,&\text{if } i\neq k\\ F_k/x_k,~& \text{if }i=k.\end{cases}\text{ and }
 y_i^{\prime}=\begin{cases} y_k^{-1}~,& i=k\\ y_iy_k^{max(b_{ki},0)}(1\bigoplus y_k)^{-b_{ki}}~,&otherwise.
\end{cases}.\nonumber
\end{eqnarray}
\end{Definition}

It can be seen that $\mu_k({\bf x},{\bf y},B)$ is also a  seed and $\mu_k(\mu_k({\bf x},{\bf y},B))=({\bf x},{\bf y},B)$.

\begin{Definition}\label{defcpattern}
  A {\bf cluster pattern} $\mathcal S$ is an assignment of a  seed  $({\bf x}_t,{\bf y}_t,B_t)$ to every vertex $t$ of the $n$-regular tree $\mathbb T_n$, such that $({\bf x}_{t^{\prime}},{\bf y}_{t^{\prime}},B_{t^{\prime}})=\mu_k({\bf x}_t,{\bf y}_t,B_t)$ for any edge $t^{~\underline{\quad k \quad}}~ t^{\prime}$.
\end{Definition}
In the sequel, we always denote by ${\bf x}_t=\{x_{1;t},\cdots, x_{n;t}\},~ {\bf y}_t=\{y_{1;t},\cdots, y_{n;t}\}$ and $B_t=(b_{ij}^t).$
 \begin{itemize}
 \item Let  $\mathcal S$ be a cluster pattern,   the {\bf cluster algebra} $\mathcal A(\mathcal S)$  associated with $\mathcal S$ is the $\mathbb {ZP}$-subalgebra of the field $\mathcal F$ generated by all cluster variables of  $\mathcal S$.
 \item If  $\mathcal S$ is cluster pattern with the coefficients in $Trop(y_1,\cdots,y_m)$, the corresponding cluster algebra $\mathcal A(\mathcal S)$ is said to be a {\bf cluster algebra of  geometric type}.
 \item If  $\mathcal S$ is cluster pattern  with the coefficients in $Trop(y_1,\cdots,y_n)$ and there exists a seed $({\bf x}_{t_0},{\bf y}_{t_0},B_{t_0})$ such that $y_{i;t_0}=y_i$ for $i=1,\cdots,n$, then  the corresponding cluster algebra  $\mathcal A(\mathcal S)$ is called a {\bf cluster algebra with principal coefficients at $t_0$}.
\end{itemize}

\begin{Theorem} \label{thmLP} Let $\mathcal A(\mathcal S)$ be a skew-symmetrizable cluster algebra  with a seed $({\bf x}_{t_0},{\bf y}_{t_0},B_{t_0})$.

(i) ( \cite[Theorem 3.1]{FZ}, Laurent phenomenon)
 Any cluster variable $x_{i;t}$ of $\mathcal A(\mathcal S)$ is a $\mathbb {ZP}$-linear combination of  Laurent monomials in ${\bf x}_{t_0}$.

(ii) (\cite{GHKK}, Positive  Laurent phenomenon)  Any cluster variable $x_{i;t}$ of $\mathcal A(\mathcal S)$ is a $\mathbb {NP}$-linear combination of Laurent monomials in ${\bf x}_{t_0}$.
\end{Theorem}

Recall that the {\bf exchange graph} ${\bf EG}(\mathcal A(\mathcal S))$ of a cluster algebra $\mathcal A(\mathcal S)$ is a graph satisfying that
\begin{itemize}
\item the set of vertices of ${\bf EG}(\mathcal A(\mathcal S))$ is in bijection with the set of clusters of $\mathcal A(\mathcal S)$;
    \item two vertices joined by an edge if and only if the corresponding two clusters differ by a single cluster variable.
\end{itemize}
For a cluster algebra  $\mathcal A(\mathcal S)$, we denote by $\mathcal X(\mathcal S)$ the set of cluster variables of  $\mathcal A(\mathcal S)$.

\begin{Definition}\label{maindef} \cite{ASV}
A cluster algebra  $\mathcal A(\mathcal S)$ is {\bf unistructural} if for any cluster algebra $\mathcal A(\mathcal S^\prime)$, $\mathcal X(\mathcal S)=\mathcal X(\mathcal S^\prime)$ implies that the two cluster algebras have the same set of clusters and ${\bf EG}(\mathcal A(\mathcal S))={\bf EG}(\mathcal A(\mathcal S^\prime))$.
\end{Definition}

\subsection{The enough $g$-pairs property}

Let $\mathcal A(\mathcal S)$ be a cluster algebra with principal coefficients at $t_0$, one can give a $\mathbb Z^n$-grading of $\mathbb Z[x_{1;t_0}^{\pm1},\cdots,x_{n;t_0}^{\pm1},y_1,\cdots,y_n]$ as follows:
$$deg(x_{i;t_0})={\bf e}_i,~deg(y_j)=-{\bf b}_j,$$
where ${\bf e}_i$ is the $i$-th column vector of $I_n$, and ${\bf b}_j$ is the $j$-th column vector of $B_{t_0}$, $i,j=1,2,\cdots,n$. As shown in \cite{FZ3} every cluster variable $x_{i;t}$ of $\mathcal A(\mathcal S)$ is homogeneous with respect to this $\mathbb Z^n$-grading. The {\bf $g$-vector} $g(x_{i;t})$ of a cluster variable $x_{i;t}$ is defined to be its degree with respect to the $\mathbb Z^n$-grading and we write $g(x_{i;t})=(g_{1i}^t,~g_{2i}^t,~\cdots,~g_{ni}^t)^{\top}\in\mathbb Z^n$. Let ${\bf x}_t$ be a cluster of $\mathcal A(\mathcal S)$,  the matrix $G_t=(g(x_{1;t}),\cdots,g(x_{n;t}))$ is called the {\bf $G$-matrix} of ${\bf x}_t$.

Denote by ${\bf x}_t^{\bf a}:=\prod\limits_{i=1}^nx_{i;t}^{a_i}$ for ${\bf a}\in \mathbb Z^n$, which is a Laurent monomial in ${\bf x}_t$. If ${\bf a}\in\mathbb N^n$, then ${\bf x}_t^{\bf a}$ is called a {\bf cluster monomial} in ${\bf x}_t$.

Clearly, any Laurent monomial ${\bf x}_t^{\bf a}$ is also homogeneous with respect to the $\mathbb Z^n$-grading. The degree of  ${\bf x}_t^{\bf a}$ is $G_t{\bf a}=:g({\bf x}_t^{\bf a})$, which is called the {\em $g$-vector of ${\bf x}_t^{\bf a}$}.

\begin{Theorem}\label{thmmonomial} \cite{GHKK,CL1}
Different cluster monomials have different $g$-vectors.
\end{Theorem}

Let $I$ be a subset of $\{1,\cdots,n\}$. We say that $(k_1,\cdots,k_s)$ is an {\bf $I$-sequence}, if $k_j\in I$ for $j=1,\cdots,s$.

\begin{Definition}
Let $\mathcal A(\mathcal S)$ be a skew-symmetrizable cluster algebra of rank $n$ with initial seed at $t_0$, and  $I=\{i_1,\cdots,i_p\}$ be a subset of $\{1,2,\cdots,n\}$.

(i) We say that a seed $({\bf x}_t,{\bf y}_t, B_t)$  of $\mathcal A(\mathcal S)$ is {\bf connected with $({\bf x}_{t_0}, {\bf y}_{t_0}, B_{t_0})$ by an $I$-sequence}, if there exists an $I$-sequence $(k_1,\cdots, k_s)$ such that $$({\bf x}_t,{\bf y}_t, B_t)=\mu_{k_s}\cdots\mu_{k_2}\mu_{k_1}({\bf x}_{t_0}, {\bf y}_{t_0}, B_{t_0}).$$

(ii) We say that a cluster ${\bf x}_t$ of $\mathcal A(\mathcal S)$ is {\bf connected with ${\bf x}_{t_0}$ by an $I$-sequence}, if there exists a seed containing the cluster ${\bf x}_t$ such that this seed is connected with
a seed containing the cluster ${\bf x}_{t_0}$ by an $I$-sequence.
\end{Definition}
Clearly, if the cluster ${\bf x}_t$ is connected with ${\bf x}_{t_0}$ by an $I$-sequence, then $x_{i;t}=x_{i;t_0}$ for $i\notin I$.

 For $I=\{i_1,\cdots,i_p\}\subseteq\{1,\cdots,n\}$, we assume that $i_1<i_2<\cdots<i_p$. Let $\pi_I: \mathbb R^n\rightarrow \mathbb R^{|I|}=\mathbb R^{p}$ be the  canonical projection given by
$\pi_I({\bf m})=(m_{i_1},\cdots,m_{i_p})^{\top}$, for ${\bf m}=(m_1,\cdots,m_n)^{\top}\in\mathbb R^n$.

\begin{Definition}
Let $\mathcal A(\mathcal S)$ be a skew-symmetrizable cluster algebra of rank $n$ with principal coefficients at $t_0$, and  $I$ be a subset of $\{1,\cdots,n\}$.

(i)  For two clusters ${\bf x}_t,{\bf x}_{t^\prime}$ of $\mathcal A(\mathcal S)$, the pair $({\bf x}_t,{\bf x}_{t^\prime})$ is called a {\bf $g$-pair along $I$}, if it satisfies the following  conditions:
\begin{itemize}
\item ${\bf x}_{t^\prime}$ is connected with ${\bf x}_{t_0}$ by an $I$-sequence and

\item for any cluster monomial ${\bf x}_t^{\bf v}$ in ${\bf x}_t$, there exists a cluster monomial ${\bf x}_{t^{\prime}}^{{\bf v}^{\prime}}$ in ${\bf x}_{t^{\prime}}$ with $v_i^{\prime}=0$ for $i\notin I$
  such that
$$\pi_I(g({\bf x}_t^{\bf v}))=\pi_I(g({\bf x}_{t^{\prime}}^{{\bf v}^{\prime}})),$$
where $g({\bf x}_t^{\bf v})$ and $g({\bf x}_{t^{\prime}}^{{\bf v}^{\prime}})$ are $g$-vectors of the cluster monomials
${\bf x}_t^{\bf v}$ and ${\bf x}_{t^{\prime}}^{{\bf v}^{\prime}}$ respectively.
\end{itemize}

(ii) $\mathcal A(\mathcal S)$ is said to have  {\bf the enough $g$-pairs property} if for  any subset $I$ of $\{1,\cdots,n\}$ and any cluster ${\bf x}_t$ of $\mathcal A(\mathcal S)$, there exists a cluster ${\bf x}_{t^\prime}$ such that $({\bf x}_t,{\bf x}_{t^\prime})$ is  a $g$-pair along $I$.
\end{Definition}

\begin{Theorem}\cite{CL1}\label{thmenough}
Any skew-symmetrizable cluster algebra $\mathcal A(\mathcal S)$ with principal coefficients at $t_0$ has  the enough $g$-pairs property.
\end{Theorem}

\subsection{Compatibility degree on the set of cluster variables}

  Let $\mathcal A(\mathcal S)$ be a skew-symmetrizable cluster algebra, and $({\bf x}_{t_0}, {\bf y}_{t_0}, B_{t_0})$ be a seed of  $\mathcal A(\mathcal S)$. By the Laurent phenomenon,  any cluster variable $x$ of $\mathcal A(\mathcal S)$ has the form of $x=\sum\limits_{{\bf v}\in V} c_{\bf v}{\bf x}_{t_0}^{\bf v}$, where $V$ is a finite subset of $\mathbb Z^n$, $0\neq c_{\bf v}\in \mathbb {ZP}$. Let $-d_{j}$ be the minimal exponent of $x_{j;t_0}$ appearing in the expansion $x=\sum\limits_{{\bf v}\in V} c_{\bf v}{\bf x}_{t_0}^{\bf v}$. Then $x$ has the form of
\begin{eqnarray}
\label{eqd}x=\frac{f(x_{1;t_0},\cdots,x_{n;t_0})}{x_{1;t_0}^{d_1}\cdots x_{n;t_0}^{d_n}},\nonumber
\end{eqnarray}
where $f\in\mathbb {ZP}[x_{1;t_0},\cdots,x_{n;t_0}]$ with  $x_{j;t_0}\nmid f$ for $j=1,\cdots,n$.
The vector

$$d^{t_0}(x) =(d_1^{t_0}(x),\cdots, d_n^{t_0}(x))^{\top}:= (d_1,\cdots, d_n)^{\top}$$ is called the {\bf denominator vector} (briefly, {\bf $d$-vector})  of the cluster variable $x$ with respect to ${\bf x}_{t_0}$.

Let $\mathcal A(\mathcal S)$ be a skew-symmetrizable cluster algebra, and $\mathcal X(\mathcal S)$ be the set of cluster variables of  $\mathcal A(\mathcal S)$. In \cite{CL1}, we proved that there exists a well-defined function $d:\mathcal X(\mathcal S)\times \mathcal X(\mathcal S)\rightarrow\mathbb Z_{\geq -1}$, which is called the {\bf compatibility degree}.  For any two cluster variables $x_{i;t}$ and $x_{j;t_0}$, the value of $d(x_{j;t_0},x_{i;t})$ is defined by the following steps:
\begin{itemize}
\item choose a cluster ${\bf x}_{t_0}$ containing the cluster variable $x_{j;t_0}$;
\item compute the $d$-vector of $x_{i;t}$ with respect to ${\bf x}_{t_0}$, say, $d^{t_0}(x_{i;t})=(d_1,\cdots,d_n)^{\top}$;
    \item $d(x_{j;t_0},x_{i;t}):=d_j$, which is called the {\bf compatibility degree of $x_{i;t}$ with respect to $x_{j;t_0}$}.
\end{itemize}
\begin{Remark}\label{rmkd}
\cite{CL1} The compatibility degree has the following properties:

(1)  The value $d(x_{j;t_0},x_{i;t})$ does not depend on the choice of ${\bf x}_{t_0}$, thus the compatibility degree function is well-defined.

(2) $d(x_{j;t_0},x_{i;t})=-1$ if and only if $d(x_{i;t},x_{j;t_0})=-1$, and if and only if $x_{j;t_0}=x_{i;t}$;

(3) $d(x_{j;t_0},x_{i;t})=0$ if and only if $d(x_{i;t},x_{j;t_0})=0$, and if and only if $x_{j;t_0}\neq x_{i;t}$ and there exists a cluster ${\bf x}_{t^\prime}$ containing both $x_{j;t_0}$ and $x_{i;t}$;

(4) By (2), (3) and $d(x_{j;t_0},x_{i;t})\geq-1$, we know that $d(x_{j;t_0},x_{i;t})\leq0$ if and only if $d(x_{i;t},x_{j;t_0})\leq0$, if and only if there exists a cluster ${\bf x}_{t^\prime}$ containing both $x_{j;t_0}$ and $x_{i;t}$.

(5) By (4), we know that $d(x_{j;t_0},x_{i;t})>0$ if and only if $d(x_{i;t},x_{j;t_0})>0$, if and only if there exists no cluster ${\bf x}_{t^\prime}$ containing both $x_{j;t_0}$ and $x_{i;t}$.
\end{Remark}

We say that $x_{i;t}$ and $x_{j;t_0}$ are {\bf $d$-compatible} if $d(x_{j;t_0},x_{i;t})\leq0$, i.e., if there exists a cluster ${\bf x}_{t^\prime}$ containing both $x_{j;t_0}$ and $x_{i;t}$. A subset $M$ of $\mathcal A(\mathcal S)$  is a  {\bf $d$-compatible set} if any two cluster variables in this set are $d$-compatible.

There is another type of compatible sets, which we call $c$-compatible sets. A subset $M$ of $\mathcal A(\mathcal S)$  is a  {\bf $c$-compatible set} if there exists a cluster ${\bf x}_{t^\prime}$ such that $M\subseteq {\bf x}_{t^\prime}$. Roughly speaking, $c$-compatibility is just compatibility with respect to clusters.

\begin{Theorem}\label{thmcomp}\cite{CL1}
Let $\mathcal A(\mathcal S)$ be a skew-symmetrizable cluster algebra, and $\mathcal X(\mathcal S)$ be the set of cluster variables of  $\mathcal A(\mathcal S)$. Then

(i) a subset $M$ of $\mathcal X(\mathcal S)$ is a $d$-compatible set if and only if it is a  $c$-compatible set, i.e., $M$ is a subset of some cluster of $\mathcal A(\mathcal S)$.

(ii) a subset $M$ of $\mathcal X(\mathcal S)$ is a maximal $d$-compatible set if and only it is a maximal $c$-compatible set, i.e., $M$ is a cluster of $\mathcal A(\mathcal S)$.
\end{Theorem}

\section{Proof of Theorem \ref{mainthm}}

\begin{Lemma}\label{lem1}
Let $\mathcal A(S)$ be a skew-symmetrizable cluster algebra with principal coefficients at $t_0$, then for  any  cluster variable $x_{i;t}$ and any initial cluster variable $x_{k;t_0}$, there exists a cluster ${\bf x}_{t^\prime}$ containing $x_{k;t_0}$ and a  Laurent monomial $F$ appearing in the Laurent expansion of $x_{i;t}$ with respect to ${\bf x}_{t^\prime}$ such that the following statements hold.

(i) The exponent of $x_{j;t^\prime}$ in $F$ are nonnegative for any $j\neq k$.

 (ii) (a) If the exponent of $x_{k;t^\prime}$ in $F$  is positive, then $x_{i;t}=x_{k;t^\prime}$;

 (b) If the exponent of $x_{k;t^\prime}$ in $F$ is zero, then $x_{i;t}\in {\bf x}_{t^\prime}$ and $x_{i;t}\neq x_{k;t^\prime}$;

(c) If $x_{i;t}\notin {\bf x}_{t^\prime}$, then the exponent of $x_{k;t^\prime}$  in $F$  is negative.
\end{Lemma}
\begin{proof}
This lemma  essentially due to Lemma \cite[Lemma 5.2]{CL1}, but with different presentation. For the convenience of the readers, we give a complete proof for this lemma.

(i). Without loss of generality, we can assume that $k=n$.
By Theorem \ref{thmenough}, $\mathcal A(S)$ has the enough $g$-pairs property, so for the cluster ${\bf x}_t$ and
$I=\{1,\cdots,n-1\}$, there exists a cluster ${\bf x}_{t^\prime}$ such that $({\bf x}_t,{\bf x}_{t^\prime})$ is a $g$-pair along $I$. Consider the Laurent expansion of $x_{i;t}$ with respect to ${\bf x}_{t^{\prime}}$, by \cite[Theorem]{CL1} this expansion has the form of

 $$x_{i;t}={\bf x}_{t^\prime}^{\bf r}(1+\sum\limits_{0\neq {\bf v}\in\mathbb N^n,~~{\bf u}\in\mathbb Z^n}c_{\bf v}{\bf y}^{\bf v}{\bf x}_{t^\prime}^{\bf u}),$$
where $c_{\bf v}\geq 0$ and ${\bf r}$ satisfies $g(x_{i;t})=G_{t^\prime}{\bf r}$.  Denote by $g(x_{i;t})=(g_1,\cdots,g_n)^\top$ and ${\bf r}=(r_1,\cdots,r_n)^\top$.

We will show that  $F={\bf x}_{t^\prime}^{\bf r}$ satisfies the requirements in Lemma \ref{lem1}.

Since $({\bf x}_t,{\bf x}_t^\prime)$ is a $g$-pair along $I=\{1,\cdots,n-1\}$, we know that for the cluster variable $x_{i;t}$ (as a cluster monomial in ${\bf x}_t$), there exists a cluster monomial ${\bf x}_{t^\prime}^{{\bf v}^\prime}$, with $v_j^\prime=0$ is zero for $j\notin I$, i.e., $v_n^\prime=0$ such that $$\pi_I(g(x_{i;t}))=\pi_I(g({\bf x}_{t^{\prime}}^{{\bf v}^\prime}))=\pi_I(G_{t^\prime}{\bf v}^\prime).$$

Because  $({\bf x}_t,{\bf x}_t^\prime)$ is a $g$-pair along $I=\{1,\cdots,n-1\}$, we know that ${\bf x}_{t^\prime}$ is connected with ${\bf x}_{t_0}$ by an $I=\{1,\cdots,n-1\}$-sequence, and the $G$-matrix of ${\bf x}_{t^\prime}$ has the form of $G_{t^{\prime}}=\begin{pmatrix}G(t^{\prime})&0\\ \ast& 1\end{pmatrix}$.
Thus  $\pi_I(g(x_{i;t}))=\pi_I(g({\bf x}_{t^{\prime}}^{{\bf v}^\prime}))=\pi_I(G_{t^\prime}{\bf v}^\prime)$ just means that
$$(g_{1},\cdots,g_{n-1})^{\top}=G(t^\prime)(v_{1}^\prime,\cdots,v_{n-1}^\prime)^{\top}.$$
By $g(x_{i;t})=(g_{1},\cdots,g_{n-1},g_{n})^{\top}=G_{t^\prime}{\bf r}$, we get that  $$(g_{1},\cdots,g_{n-1})^{\top}=G(t^\prime)(r_{1},\cdots,r_{n-1})^{\top}.$$
It is known  that $det(G(t^\prime))=det(G_{t^{\prime}})=\pm1$. Thus we get that $$(r_{1},\cdots,r_{n-1})^{\top}=(v_{1}^\prime,\cdots,v_{n-1}^\prime)^{\top}\in\mathbb N^{n-1}.$$

(ii). (a) and (b): If $r_{n}\geq 0$, then ${\bf r}\in \mathbb N^n$, and ${\bf x}_{t^\prime}^{{\bf r}}$ is a cluster monomial in ${\bf x}_{t^\prime}$ having the same $g$-vector  with the cluster variable $x_{i;t}$. By Theorem \ref{thmmonomial}, we get that $x_{i;t}={\bf x}_{t^\prime}^{{\bf r}}$. Then by \cite[Lemma 5.1]{CL1}, $x_{i;t}$ is a cluster variable in ${\bf x}_{t^\prime}$. More precisely, if $r_{n}>0$, then $x_{i;t}=x_{n;t^{\prime}}$. If $r_{n}=0$, then $x_{i;t}=x_{j;t^\prime}$ for some $j\neq n$.

 (c) follows from (a) and (b).
\end{proof}

The above lemma is about cluster algebras with principal coefficients, and we turn it into cluster algebras with general coefficients in the following proposition.
\begin{Proposition}\label{mainpro}
Let $\mathcal A(S)$ be a skew-symmetrizable cluster algebra, then for any two cluster variables $x_{i;t}$ and $x_{k;t_0}$,  there exists a cluster ${\bf x}_{t^\prime}$ containing $x_{k;t_0}$ and a  Laurent monomial $F$ appearing in the Laurent expansion of $x_{i;t}$ with respect to ${\bf x}_{t^\prime}$ such that the following statements hold.

(i) The exponent of $x_{j;t^\prime}$ in $F$ are nonnegative for any $j\neq k$.

 (ii) (a) If the exponent of $x_{k;t^\prime}$ in $F$  is positive, then $x_{i;t}=x_{k;t^\prime}$;

 (b) If the exponent of $x_{k;t^\prime}$ in $F$ is zero, then $x_{i;t}\in {\bf x}_{t^\prime}$ and $x_{i;t}\neq x_{k;t^\prime}$;

(c) If $x_{i;t}\notin {\bf x}_{t^\prime}$, then the exponent of $x_{k;t^\prime}$  in $F$  is negative.
\end{Proposition}

\begin{proof}
Consider the cluster algebra $\mathcal A(S^{pr})$ with principal coefficients at $t_0$ and having the same exchange matrix with $\mathcal A(S)$ at $t_0$. Then Lemma \ref{lem1} holds for $\mathcal A(S^{pr})$, then the result follows from \cite[Theorem 3.7]{FZ3}.
\end{proof}

{\em The proof of Theorem \ref{mainthm}:}
Let  $\mathcal A(\mathcal S)$ and $\mathcal A(\mathcal S^\prime)$ be two skew-symmetrizable cluster algebras having the same set of cluster variables, i.e., $\mathcal X(\mathcal S)=\mathcal X(\mathcal S^\prime)$. We need to show that  $\mathcal A(\mathcal S)$ and $\mathcal A(\mathcal S^\prime)$ have the same set of clusters and ${\bf EG}(\mathcal A(\mathcal S))={\bf EG}(\mathcal A(\mathcal S^\prime))$.

We first show that  $x,z\in \mathcal X(\mathcal S)=\mathcal X(\mathcal S^\prime)$ are $d$-compatible in $\mathcal A(\mathcal S)$ if and only if  they are $d$-compatible  in $\mathcal A(\mathcal S^\prime)$.

Denote by ${\bf x}_{t}=\{x_{1;t},\cdots,x_{n;t}\}$ the cluster of $\mathcal A(\mathcal S)$ at the vertex $t\in\mathbb T_n$ and by ${\bf z}_{u}=\{z_{1;u},\cdots,z_{m;u}\}$ the cluster of $\mathcal A(\mathcal S^\prime)$ at the vertex $u\in\mathbb T_m$. (Note that there is no need for us to assume that $m=n$.)

Let $x,z\in  \mathcal X(\mathcal S)=\mathcal X(\mathcal S^\prime)$ be two cluster variables, which are $d$-compatible  in $\mathcal A(\mathcal S^\prime)$. Assume by contradiction that $x$ and $z$ are not $d$-compatible in  $\mathcal A(\mathcal S)$, i.e., there exists no cluster ${\bf x}_{t}$ of $\mathcal A(\mathcal S)$ containing both $x$ and $z$. For the cluster variables $z$ and $x$, applying Proposition \ref{mainpro} for $\mathcal A(\mathcal S)$, there exists a cluster ${\bf x}_{t^\prime}$ of $\mathcal A(\mathcal S)$ containing $x$ (say, $x=x_{1;t^\prime}$), such that  in the Laurent expansion of $z$ with respect to ${\bf x}_{t^\prime}$, there exists a  Laurent monomial $F$ satisfying that the exponent of $x_{j;t^\prime}$ in $F$ are nonnegative for any $j\neq 1$ and the exponent of $x=x_{1;t^\prime}$ in $F$ is negative (because there exists no cluster  of $\mathcal A(\mathcal S)$ containing both $x$ and $z$, in particular, $z\notin {\bf x}_{t^\prime}$). We can assume that  $F=cx_{1;t^\prime}^{-v_1}\prod\limits_{i=2}^nx_{i;t^\prime}^{v_i}$ with $v_1>0$,\;$v_2,\cdots,v_n\geq 0$ and $0\neq c\in\mathbb {NP}$ for some semifield $\mathbb P$. Thus the Laurent expansion of $z$ with respect to ${\bf x}_{t^\prime}$ can be written as

$$z=F+\tilde F(x_{1;t^\prime},\cdots,x_{n;t^\prime})=cx_{1;t^\prime}^{-v_1}\prod\limits_{i=2}^nx_{i;t^\prime}^{v_i}+\tilde F(x_{1;t^\prime},\cdots,x_{n;t^\prime}),$$
where $\tilde F$ is a Laurent polynomial with positive coefficients.

Since $x=x_{1;t^\prime}$ and $z$ are $d$-compatible  in $\mathcal A(\mathcal S^\prime)$, there exists a cluster ${\bf z}_u$ of $\mathcal A(\mathcal S^\prime)$ such that ${\bf z}_u$ containing both $x=x_{1;t^\prime}$ and $z$. Without loss of generality, we can assume that $x=x_{1;t^\prime}=z_{1;u}$ and $z=z_{2;u}$.
Consider the Laurent expansion of $x_{i;t^\prime}$ with respect to ${\bf z}_u$,

$$x_{i;t^\prime}=\frac{g_i(z_{1;u},\cdots,z_{m;u})}{z_{1;u}^{d_{1i}}\cdots z_{m;u}^{d_{ni}}},$$
where $g_i$ is a polynomial in $z_{1;u},\cdots,z_{m;u}$ with positive coefficients and $z_{l;u}\nmid g_i$ for any $l$.
By Remark \ref{rmkd}, and $x_{i;t^\prime}\neq x_{1;t^\prime}=x=z_{1;u}$ for any $i=2,\cdots,n$, we know that $d_{1i}\geq 0$ for $i=2,\cdots,n$. So for each $i=2,\cdots,n$, there exists a Laurent monomial $G_i$ appearing in the expansion of $x_{i;t^\prime}$ with respect to ${\bf z}_u$ such that the exponent of $z_{1;u}=x$ is nonpositive in $G_i$, i.e., $x_{i;t^\prime}$ has the form of
\begin{eqnarray}
x_{i;t^\prime}=G_i+\tilde G_i(z_{1;u},\cdots,z_{m;u})=c_iz_{1;u}^{-a_{1i}}\prod\limits_{l=2}^mz_{l;u}^{a_{li}}+\tilde G_i(z_{1;u},\cdots,z_{m;u}),\nonumber
\end{eqnarray}
where $\tilde G_i$ is a Laurent polynomial with positive coefficients, and $a_{1i}\geq 0$, $0\neq c_i\in \mathbb {NP}$ for some semifield $\mathbb P$.

Substituting $x_{1;t^\prime}=x=z_{1;u}$ and $x_{i;t^\prime}=c_iz_{1;u}^{-a_{1i}}\prod\limits_{l=2}^mz_{l;u}^{a_{li}}+\tilde G_i(z_{1;u},\cdots,z_{n;u})$ for $i\geq 2$ into $$z=cx_{1;t^\prime}^{-v_1}\prod\limits_{i=2}^nx_{i;t^\prime}^{v_i}+\tilde F(x_{1;t^\prime},\cdots,x_{n;t^\prime}),$$
we obtain the Laurent expansion of $z=z_{2;u}$ with respect to ${\bf z}_u$, which has the form of
\begin{eqnarray}
z_{2;u}=z&=&cz_{1;u}^{-v_1}\prod_{i=2}^n(c_iz_{1;u}^{-a_{1i}}\prod\limits_{l=2}^mz_{l;u}^{a_{li}})^{v_i}+R(z_{1;u},\cdots,z_{m;u})\nonumber\\
&=&c\prod_{i=2}^nc_iz_{1;u}^{-(v_1+v_2a_{12}+\cdots v_na_{1n})}\prod\limits_{l=2}^mz_{l;u}^{v_2a_{l2}+\cdots+v_na_{ln}}+
R(z_{1;u},\cdots,z_{m;u}),\nonumber
\end{eqnarray}
where $R$ can be written as $R=\frac{r_1(z_{1;u},\cdots,z_{m;u})}{r_2(z_{1;u},\cdots,z_{m;u})}$ with $r_1,r_2\in \mathbb {NP}[z_{1;u},\cdots,z_{m;u}]$ for some semifield $\mathbb P$.
Thus we get that
\begin{eqnarray}\label{eqnmain}
z_{2;u}-c\prod_{i=2}^nc_iz_{1;u}^{-(v_1+v_2a_{12}+\cdots v_na_{1n})}\prod\limits_{l=2}^mz_{l;u}^{v_2a_{l2}+\cdots+v_na_{ln}}=
\frac{r_1(z_{1;u},\cdots,z_{m;u})}{r_2(z_{1;u},\cdots,z_{m;u})}.
\end{eqnarray}
Note that there is a $\mathbb R$-algebra homorphism  $\varphi: \mathbb{RP}(z_{1;u},\cdots,z_{m;u})\rightarrow \mathbb R(z_{1;u},\cdots,z_{m;u})$ given by $$\varphi (a)=\begin{cases} 1&\text{if }a\in\mathbb P \;;\\ z_{p;u}&\text{if }a=z_{p;u}\;\text{ for some }p=1,\cdots,m.\end{cases}$$
So the equality (\ref{eqnmain}) in $\mathbb{RP}(z_{1;u},\cdots,z_{m;u})$ induces a equality in $\mathbb R(z_{1;u},\cdots,z_{m;u})$ by the action of the homorphism $\varphi$. The following is the new equality.
\begin{eqnarray}\label{eqnmain2}
z_{2;u}-\varphi(c\prod_{i=2}^nc_i)z_{1;u}^{-(v_1+v_2a_{12}+\cdots v_na_{1n})}\prod\limits_{l=2}^mz_{l;u}^{v_2a_{l2}+\cdots+v_na_{ln}}=
\frac{\varphi(r_1)(z_{1;u},\cdots,z_{m;u})}{\varphi(r_2)(z_{1;u},\cdots,z_{m;u})},
\end{eqnarray}
where $\varphi(r_1),\varphi(r_2)\in\mathbb N[z_{1;u},\cdots,z_{m;u}]$ and $\varphi(c\prod_{i=2}^nc_i)\geq 1$.

Since $v_1>0$ and $v_i, a_{1i}\geq 0$, we get that $v_1+v_2a_{12}+\cdots v_na_{1n}>0$. Take $z_{1;u}=1/2=2^{-1}$ and $z_{2;u}=z_{3;u}=\cdots=z_{m;u}=1$, the left side of the equality (\ref{eqnmain2}) is
$$1-\varphi(c\prod_{i=2}^nc_i)2^{(v_1+v_2a_{12}+\cdots v_na_{1n})}\leq 1-2^{(v_1+v_2a_{12}+\cdots v_na_{1n})}<0,$$
 but the right side of the equality (\ref{eqnmain2}) is  nonnegative in this case, by  $\varphi(r_1),\varphi(r_2)\in\mathbb N[z_{1;u},\cdots,z_{m;u}]$. This is a contradiction. So if $x$ and $z$ are $d$-compatible  in $\mathcal A(\mathcal S^\prime)$, they must be $d$-compatible in $\mathcal A(\mathcal S)$.

Similarly, we can show that if  $x$ and $z$ are $d$-compatible  in $\mathcal A(\mathcal S)$, they must be $d$-compatible in $\mathcal A(\mathcal S^\prime)$.

Thus  $x,z\in \mathcal X(\mathcal S)=\mathcal X(\mathcal S^\prime)$ are $d$-compatible in $\mathcal A(\mathcal S)$ if and only if  they are $d$-compatible  in $\mathcal A(\mathcal S^\prime)$. So a subset $M\subseteq \mathcal X(\mathcal S)=\mathcal X(\mathcal S^\prime)$ is a maximal $d$-compatible set in $\mathcal A(\mathcal S)$ if and only if  it is a maximal $d$-compatible set in $\mathcal A(\mathcal S^\prime)$. That is, a subset $M\subseteq \mathcal X(\mathcal S)=\mathcal X(\mathcal S^\prime)$ is a cluster in $\mathcal A(\mathcal S)$ if and only if  it is a cluster in $\mathcal A(\mathcal S^\prime)$, by Theorem \ref{thmcomp}. Hence, $\mathcal A(\mathcal S)$ and $\mathcal A(\mathcal S^\prime)$ have the same set of clusters.

Since the exchange graph of a cluster is uniquely determined by the set of clusters, we know that ${\bf EG}(\mathcal A(\mathcal S))={\bf EG}(\mathcal A(\mathcal S^\prime))$. This  completes the proof. \hfill$\Box$

\vspace{2mm}
Let $\mathcal A(\mathcal S)$ be a cluster algebra with coefficients in $\mathbb{ZP}$. A {\bf cluster automorphism} of $\mathcal A(\mathcal S)$ is a $\mathbb {ZP}$-automorphism of the algebra $\mathcal A(\mathcal S)$ mapping a cluster to a cluster and commuting with mutations.

\begin{Corollary}
Let $\mathcal A(\mathcal S)$ be a skew-symmetrizable cluster algebra. Then $f:\mathcal A(\mathcal S)\rightarrow \mathcal A(\mathcal S)$ is a cluster automorphism if and only if $f$ is an automorphism of the ambient field $\mathcal F$ which restricts to a permutation of the set of cluster variables.
\end{Corollary}
\begin{proof}
It is known that this result is true for unistructural cluster algebras by \cite[Theorem 1.4]{ASV}.  Then it follows from Theorem \ref{mainthm}.
\end{proof}

We know that the definition of compatibility degree function on the set of cluster variables mainly depends on the cluster structure of a cluster algebra. By Theorem \ref{mainthm}, the cluster structure of a cluster algebra is uniquely determined by the set of cluster variables, so the compatibility degree function is an  intrinsic function on the set of cluster variables. It would be interesting to give a geometric or categorial explanation for compatibility degree function in general. Note that for cluster algebras of finite type or  from surfaces, the compatibility degree function has nice explanation and one can refer to \cite{FZ0,FZ1,CP,Z,FST}.

\vspace{5mm}
{\bf Acknowledgements:}\; This project is supported by the National Natural Science Foundation of China (No.11671350 and No.11571173).

%\end{CJK*}

\end{document}